\documentclass[twoside,a4paper,reqno,11pt]{amsart} 
\usepackage[top=30mm,right=30mm,bottom=30mm,left=30mm]{geometry}
\usepackage[utf8]{inputenc}
\usepackage{amsthm}
\usepackage{amsfonts}
\usepackage{amsmath}
\usepackage{ amssymb }
\usepackage{url}
\usepackage{todonotes}
\usepackage{hyperref} 
\usepackage{enumerate}%
\usepackage{comment}
\usepackage[foot]{amsaddr}
\newtheorem{theorem}{Theorem}[section]
\newtheorem{proposition}[theorem]{Proposition}
\newtheorem{corollary}[theorem]{Corollary}
\newtheorem*{theorem*}{Theorem}
\newtheorem{manualtheoreminner}{Theorem}
\newenvironment{manualtheorem}[1]{%
  \renewcommand\themanualtheoreminner{#1}%
  \manualtheoreminner
}{\endmanualtheoreminner}

\newtheorem{lemma}[theorem]{Lemma}
\theoremstyle{definition}
\newtheorem{definition}[theorem]{Definition}
\newtheorem*{definition*}{Definition}
\newtheorem*{proposition*}{Proposition}

\newtheorem*{question*}{Question}
\newtheorem{question}[theorem]{Question}

\theoremstyle{remark}
\newtheorem{remark}[theorem]{Remark}
\newtheorem{example}[theorem]{Example}
\numberwithin{equation}{section}

\title{On finite $p$-groups with powerful subgroups}
\author{James Williams}
\address{School of Mathematics, University of Bristol, Bristol BS8 1UG, UK, and Heilbronn
Institute for Mathematical Research, Bristol, UK}
\email{j.l.i.williams@bristol.ac.uk}

\date{\today}

\begin{document}

\maketitle
\begin{abstract}
In this paper we investigate the structure of finite $p$-groups with the property that every subgroup of index $p^i$ is powerful for some $i$. For odd primes $p$, we show that under certain conditions these groups must be potent. Then, motivated by a question of Mann, we investigate in detail the case when all maximal subgroups are powerful. We show that for odd $p$ any finite $p$-group $G$ with all maximal subgroups powerful has a regular power structure - with precisely one exceptional case which is a $3$-group of maximal class and order $81$. To show this counterexample is unique we use a computational approach. We briefly discuss the case $p=2$ and some generalisations.
\end{abstract}

\section{Introduction}
Loosely speaking, the purpose of this paper is to study $p$-groups which contain large subgroups with some desirable properties and to determine whether these properties are enjoyed by the group itself. In doing so we address questions posed by Mann and Xu. 

More precisely, we study finite $p$-groups $G$ such that every subgroup of index $p^i$ is powerful. Powerful $p$-groups share many properties with abelian $p$-groups, in particular products and powers behave well in powerful $p$-groups. We show that under certain conditions these desirable properties of the subgroups of index $p^i$ extend to the group itself.

The study of $p$-groups is important, as they occur in many problems and have connections to many areas of mathematics. However, finite $p$-groups are known to be complicated. The number of $p$-groups grows very quickly and any hope of a strong classification (akin to the classification of finite simple groups) is thought to be hopeless. However, what has been a successful approach is to understand large families of $p$-groups which are in some sense ``well behaved".

To state our results precisely we need to remind the reader of the following terminology.

Let $G$ be a finite $p$-group. 
\begin{align*}
    &\Omega_{\{k\}}(G)= \{ g \in G \mid g^{p^{k}}=1 \} \quad &\Omega_{k}(G) = \langle \Omega_{\{k\}}(G) \rangle. \\
    &\mho_{\{k\}}(G)=\{ g^{p^{k}} \mid g \in G \} \quad &\mho_{k}(G) = \langle \mho_{\{k\}}(G)\rangle.
\end{align*}

We call $\Omega_{k}(G)$ the $k$th Omega subgroup of $G$. There is a sense in which the subgroups $\mho_{k}(G)$ and $\Omega_{k}(G)$ are dual to each other, hence $\mho_{k}(G)$ is referred to as the $k$th Agemo subgroup of $G$ (Agemo is Omega backwards). 

For $p$-groups in general, we cannot guarantee that $\Omega_{\{k\}}(G)=\Omega_{k}(G)$, nor that $\mho_{\{k\}}(G)=\mho_{k}(G)$. In words, the product of two elements of order at most $p^k$ can have order greater than $p^k$. Similarly the product of two $p^{k}$th powers need not be a $p^{k}$th power. 

If we turn our attention to a well understood family of finite $p$-groups - the abelian groups - we notice that they satisfy the following three conditions. 

\begin{align}
 \mho_{i}(G) &=\{g^{p^{i}} \mid g \in G \}. \label{eqn pth power} \\
      \Omega_{i}(G) &= \{g\in G \mid o(g) \leq p^{i} \}. \label{eqn omega conditon} \\
 |G:\mho_{i}(G)| &= | \Omega_{i}(G)| \label{eqn index condition}.
\end{align}

In \cite{phallcontribtothetheoryofgroupsofprimepowerorder} Hall showed that a large family of finite $p$-groups known as regular $p$-groups also satisfy these properties. Motivated by this, we say that any group which satisfies these three properties enjoys a \textit{regular power structure}.

The work of Hall in \cite{phallcontribtothetheoryofgroupsofprimepowerorder} initiated the study of power structure in finite $p$-groups. Since then, many families of $p$-groups have been shown to have a regular power structure, and even more have been shown to enjoy some subset of these properties. 

One such family of groups known to have a regular power structure for odd primes $p$ are the powerful $p$-groups. This family of groups was introduced in the paper \cite{powpgroups1} by Lubotzky and Mann, and has seen widespread applications and generated a huge amount of further research and generalisations (for instance \textit{potent} $p$-groups \cite{potentpgroups} ). Therefore it is natural to ask, if all maximal subgroups of a finite $p$-group $G$ are powerful, does $G$ enjoy any of the same nice properties. This then leads  naturally to the generalisation of all subgroups of index $p^i$ being powerful, and motivates the following definition.
\begin{definition*}
Let $G$ be a finite $p$-group. We say that $G$ is an $\mathcal{M}_i$ group if all subgroups of index $p^i$ in $G$ are powerful.
\end{definition*}

There is a family of finite $p$-groups called \textit{potent} $p$-groups which can be thought of as a generalisation of powerful $p$-groups. (We  recall their definition in Definition \ref{defn potent p-group}.) In \cite{arganbrightpowcommstruc} Arganbright began the study of what we now call potent $p$-groups. Their theory is developed further by  Gonz\'{a}lez-S\'{a}nchez and Jaikin-Zapirain in \cite{potentpgroups}, where among other things it is shown that for odd primes $p$, potent $p$-groups have a regular power structure. 

For our first main result, we prove the following theorem which says that under certain conditions an $\mathcal{M}_i$ group is a potent $p$-group.

\begin{manualtheorem}{A}
\label{intro theorem i <= p-3 then potent}
Let $p$ be an odd prime and $G$ be a finite $p$-group. If $i \leq p-3$ and $G$ is an $\mathcal{M}_{i}$ group, then $G$ is a potent $p$-group.
\end{manualtheorem}
This theorem captures a relationship between the prime $p$, the index $p^i$, and the power structure properties of the group.  In a similar way, we can capture a relationship with the minimal number of generators of the group. Letting $d(G)$ be the minimal number of generators of $G$, we also prove the following result.

\begin{manualtheorem}{B}
\label{intro theorem rank then powerful}
Let $G$ be a finite $p$-group. If $G$ is a $\mathcal{M}_i$ group with $d(G) \geq i+2$, then $G$ is powerful.
\end{manualtheorem}

Turning our attention to the regular power structure properties, as both powerful and potent $p$-groups have regular power structure, we deduce that for $p>3$ all $\mathcal{M}_1$ $p$-groups have a regular power structure. 

The main part of the paper focuses on the study of $\mathcal{M}_1$ groups. We recall a question of Mann from \cite[Question 68]{berkovichvol1} which is a motivation for this:
\begin{question*}[Mann] Study the $p$-groups, $p>2$, all of whose (a) subgroups of index $p$ (of index $p^2$) are powerful (b) subgroups of indices $p$ and $p^2$ are powerful.
\end{question*}

We remark that to the best of our knowledge there are no published results about $p$-groups all of whose maximal subgroups are powerful. 

For $p>3$ it follows from Theorem \ref{intro theorem i <= p-3 then potent} that $\mathcal{M}_1$ groups are potent and thus have a regular power structure. As often happens, $p=3$ proves to be the challenging case.

We show that for $p=3$, $\mathcal{M}_1$ groups satisfy condition \ref{eqn omega conditon} and condition \ref{eqn pth power}. We believe our approach to establishing condition \ref{eqn pth power} is unique and may have other uses; to prove it we reduce to a case where we can apply a result of Mann from \cite{mannpowstructure1}. 

Our most surprising result comes when studying condition \ref{eqn index condition}. In this case we find that $\mathcal{M}_1$ $3$-groups satisfy the condition, except for one specific $3$-group of order $81$ and maximal class. Proving the uniqueness of this example is a large part of the paper and our approach makes use of the computer algebra package GAP \cite{GAP4.10.2}.

Our main theorem is stated below.
\begin{manualtheorem}{C}
\label{intro theorem main}
Let $G$ be an $\mathcal{M}_1$ $p$-group and $p$ an odd prime. Then $G$ has a regular power structure, unless $G$ is the unique $\mathcal{M}_1$ group of order $81$ and maximal class, in which case condition \ref{eqn index condition} fails to hold.
\end{manualtheorem}

We briefly discuss some possible generalisations of this theorem in Remark \ref{further remark on what happens to pow structure if i is increased}.

A large motivation for this work is to try to better understand the families of $3$-groups with regular power structure. In \cite{xusurveypaper}, Xu suggests that the problem of determining all irregular $p$-groups with a regular power structure is likely to be very difficult, and so poses some problems with restrictions such as determining the $2$-generator $3$-groups with regular power structure.  

The work in this paper is then a contribution towards better understanding the $3$-groups with regular power structure. In particular for $3$-groups of order strictly larger than $81$ we have found a new family of $3$-groups with regular power structure.

\vspace{3mm}

\noindent \textbf{Notation.}  Our notation is standard. For a group $G$, we denote the Frattini Subgroup by $\Phi(G)$. Our commutators are left normed. We define $\gamma_{1}(G)=G$, and then for $i>1$ $\gamma_{i}(G)=[\gamma_{i-1}(G),G]$.  We use bar notation to denote images under a quotient map and in all cases we shall we make it explicit what we are quotienting by. We denote the minimal number of generators of $G$ by $d(G)$ and the exponent of $G$ by $\exp G$.

\section*{Acknowledgements}
The author would like to thank Dr Gareth Tracey and Dr Tim Burness for their helpful advice and comments on this work. 

\section{Preliminaries}
\label{section preliminaries}
For the convenience of the reader we collect here the basic definitions and results that will be used in this paper. A reader familiar with powerful and potent $p$-groups may skip this section.

Powerful $p$-groups were introduced by Lubotzky and Mann in \cite{powpgroups1}. We recall their definition below. 
\begin{definition}
For odd primes $p$, a $p$-group $G$ is \textit{powerful} if $[G,G] \leq \mho_{1}(G)$. A $2$-group $G$ is \textit{powerful} if $[G,G] \leq \mho_{2}(G)$.
\end{definition}

The property of being powerful is preserved by quotients (see \cite[Lemma 2.2(i)]{analyticpropgroups}), but not necessarily when taking subgroups.

In this paper we make use of several properties of powerful $p$-groups, often without explicit mention. 

\begin{proposition}
Let $G$ be a powerful $p$-group with $p$ odd and $k, j$  non-negative integers. Then the following statements hold:
\begin{enumerate}[(i)]
    \item $[\mho_{k}(G),G] \leq \mho_{k+1}(G).$
    \item $\gamma_{k}(G) \leq \mho_{k-1}(G).$
    \item $\mho_{j}(\mho_{k}(G)) = \mho_{j+k}(G).$
\end{enumerate}
\end{proposition}
For a proof of the above see \cite[Theorem 11.10]{khukhropautfinitepgrps}.

\begin{proposition}
\label{proposition G= a1,..,ar then G^p = a_1^p, a_r^p}
Let $G= \langle a_1, \dots, a_r \rangle$ be a powerful $p$-group. Then $\mho_{1}(G)= \langle a_1^{p}, \dots, a_r^{p} \rangle$. 
\end{proposition}
For a proof of this  see \cite[Theorem 11.11]{khukhropautfinitepgrps}. For an introduction to the theory of powerful $p$-groups, we recommend Chapter 11 of \cite{khukhropautfinitepgrps}.

\begin{theorem}
\label{theorem powerful p-groups regular power structure}
For odd primes $p$, powerful $p$-groups have a regular power structure.
\end{theorem}
The first power structure condition was established in \cite[Proposition 1.7]{powpgroups1}. The second and third power structure conditions  were first established in \cite[Theorem 3.1]{wilsonpowerstrucpowpgroups}, where the proof of the third condition relied on a result of H\'ethelyi and L\'evai \cite[Corollary 2]{onelemsorderpinpowpgroups}. A shorter proof of conditions \ref{eqn omega conditon} and \ref{eqn index condition} were given by Fern\'{a}ndez-Alcober in \cite{omegasubgroupsofpowerfulpgroups}, see Theorem 1(iii) and Theorem 4. We note also that an independent proof of the fact that $|\Omega_{\{k\}}(G)| = |G: \mho_{k}(G)|$ for a powerful $p$-group $G$ was given by Mazur in \cite[Theorem 1]{onpowersinpowerfulpgroupsmazur}. \par

A generalisation of powerful $p$-groups is the notion of a potent $p$-group. 
\begin{definition}
\label{defn potent p-group}
For an odd prime $p$, a finite $p$-group $G$ is said to be potent if $\gamma_{p-1}(G) \leq \mho_{1}(G)$. A finite $2$-group $G$ is said to be \textit{potent} if $[G,G] \leq \mho_{2}(G)$.
\end{definition}
Notice that for $p=2$ and $p=3$ the notions of potent and powerful coincide. 

In \cite{potentpgroups} it is proved that for odd primes potent $p$-groups also have regular power structure. 
\begin{theorem}
\label{theorem potent p groups regular power structure}
For odd primes $p$, potent $p$-groups have a regular power structure.
\end{theorem}

Note that what we refer to as regular power structure is called \textit{power abelian} in \cite{potentpgroups}. 

In \cite{phallcontribtothetheoryofgroupsofprimepowerorder} P. Hall introduced \textit{regular $p$-groups}.  
\begin{definition}
We say that a $p$-group $G$ is \textit{regular} if for any two elements $a, b \in G$ and any $n$ we have that:
$$(ab)^{p^{n}} = a^{p^{n}}b^{p^{n}}S_1^{p^{n}} \cdots S_{t}^{p^{n}},$$
with $S_1, \dots, S_t$ appropriate elements from the commutator subgroup of the group generated by $a$ and $b$.
\end{definition}
In \cite{phallcontribtothetheoryofgroupsofprimepowerorder} P. Hall showed, amongst other things, that regular $p$-groups satisfy the regular power structure properties. For a textbook exposition of this see \cite[Theorem 7.2 (b), (c), (d)]{berkovichvol1}.

There are many known conditions which imply that a $p$-group $G$ is regular. It is clear from the definition that  $p$-groups of exponent $p$ are regular. It is also the case that $p$-groups whose nilpotency class $c$ is strictly less than $p$ are regular (\cite[Theorem 7.1(b)]{berkovichvol1}). 

We will use these facts throughout this paper, often without explicit mention. 

At one point in this paper, we will make use of Fitting's Lemma to bound the nilpotency class of a group.

\begin{lemma}[Fitting]
\label{lemma fitting} Let $G$ be a group and $M$ and $N$ normal nilpotent subgroups of $G$ with nilpotency classes $a$ and $b$ respectively. Then $MN$ is also a nilpotent normal subgroup of $G$, with nilpotency class at most $a+b$.
\end{lemma}
For a proof of Lemma \ref{lemma fitting} see \cite[Theorem 10.3.2]{halltheoryofgroupsbook}.

In Section \ref{section index property p=3} we will attempt to bound the order of some groups. In doing so, we end up considering the order of a $2$-generator group of exponent $3$. We make use of the fact that the largest $2$-generator group of exponent $3$ has order $27$ (See Chapter 18 of \cite{halltheoryofgroupsbook} for an introduction to the Burnside problem and proof of a more general version of this fact).

Finally, we would like to remind the reader that the group $\mho_{n}(G)$ is a \textit{verbal} subgroup and behaves nicely when we take quotients. In particular we have that $$ \mho_{n}\left(\frac{G}{N}\right) \cong \frac{\mho_{n}(G)N}{N}.$$ 

\subsection{A note on GAP code}
In Section \ref{section index property p=3} we will reduce  problems about $3$-groups, to problems about $3$-groups of specific small orders (at most $3^7$). As the groups of order $3^n$ for $n \in \{1, \dots, 7\}$ have been classified and are readily available in computer algebra software such as GAP \cite{GAP4.10.2} through the SmallGroupsLibrary \cite{SmallGrp1.3}, we are then able to concretely deal with these problems. 

Although some of these cases and problems could surely be dealt with directly through various ad hoc arguments, we believe there is a significant aesthetic benefit to reducing to a finite problem and then having the computer deal with it uniformly in one clean sweep. 

For the benefit of the reader, we include alongside this manuscript our GAP code to verify the claims made in Section \ref{section index property p=3}.

\section{On $\mathcal{M}_{i}$ groups}
One of our motivations for this work is to study $p$-groups all of whose maximal subgroups are powerful, these are the $\mathcal{M}_1$ groups. However to begin we study the more general case of $\mathcal{M}_i$ groups; these are the finite $p$-groups for which all subgroups of index $p^i$ are powerful.

In this paper we will often use the fact that if $G$ is an $\mathcal{M}_1$ group, then so are its quotients. This is easily seen, recalling that the property of being powerful is preserved under quotients.

\subsection{$\mathcal{M}_i$ groups and potency}\label{subsection mi groups and potency}

Fundamentally, the behaviour of $p$-groups is determined by the interaction of powers and commutators. For an $\mathcal{M}_i$ group, we are ensuring that all subgroups of index $p^i$ have a good behaviour with regards to the interactions of powers and commutators. It is natural to ask, what does this mean for the group itself.

\begin{proposition}
\label{prop G exponent p mi group class at ost i+1}
Let $G$ be an $\mathcal{M}_i$ group of exponent $p$. Then the nilpotency class of $G$ is at most $i+1$.
\end{proposition}
\begin{proof}
If the group is abelian the result is clear, hence we assume that $G$ is not abelian. All subgroups in $G$ of index $p^i$ are powerful of exponent $p$, and so they are abelian. Let $j$ be the smallest integer such that all subgroups of index $p^j$  in $G$ are abelian, and note that $j \geq i$. Then there is some subgroup $H$ in $G$ with index $p^{j-1}$ and $H$ is not abelian. However every maximal subgroup of $H$ is abelian, and so $H$ is a minimal non-abelian $p$-group of exponent $p$. As  $H$ is minimal non-abelian it follows it must be $2$-generator and have nilpotency class $2$ with $|[H,H]|=p$. Thus $|H| \leq p^3$.  (See \cite[Exercise 8a]{berkovichvol1} for a classification of minimal non-abelian $p$-groups).

Therefore the order of $G$ is at most $p^3 \cdot p^{i-1}=p^{i+2}$ and so the nilpotency class of $G$ is at most $i+1$ as required.
\end{proof}
Notice that this means an $\mathcal{M}_i$ $p$-group of exponent $p$ is a potent $p$-group.

\begin{corollary}
\label{corollary i leq p-3 an M_i group is potent}
 Let $p>2$. If $i \leq p-3$ and $G$ is an $\mathcal{M}_{i}$ group, then $G$ is a potent $p$-group.
 \end{corollary}
 \begin{proof}
 We prove this by induction on the order of $G$. The result is clearly true for groups of order $p$. Now consider some $\mathcal{M}_i$ group $G$ and suppose that the claim holds for all groups of smaller order. 
 If $G$ has exponent $p$, the result follows from Proposition \ref{prop G exponent p mi group class at ost i+1}. Now suppose that the exponent of $G$ is greater than $p$. Then $\bar{G}=G/\mho_{1}(G)$ is an $\mathcal{M}_i$ group of order strictly less than $G$ and so by induction $\gamma_{p-1}(\bar{G}) \leq \mho_{1}(\bar{G})$, and it then follows that $\gamma_{p-1}(G) \leq \mho_{1}(G)$ as required.
 \end{proof}

This establishes Theorem \ref{intro theorem i <= p-3 then potent} from the introduction. We remark that when $p=2$, the definition of potent and powerful coincide. We can readily find examples of $2$-groups for which all maximal subgroups are powerful (even abelian) and yet the group itself is not powerful - for example the dihedral and quaternion groups of order $8$.

\subsection{$\mathcal{M}_i$ groups and minimal number of generators.}
\label{subsection mi groups and rank}
In this section we ask ourselves, for a given $\mathcal{M}_i$ group $G$, is there some condition on the number of generators of $G$ that will ensure desirable properties for our group $G$. 

\begin{proposition}
\label{proposition mi d(G) >= i+2 then potent}
Let $G$ be a finite $p$-group. If $G$ is a $\mathcal{M}_i$ group with $d(G) \geq i+2$, then $G$ is powerful.
\end{proposition}
\begin{proof}
Let $\{a_1, \dots, a_r \}$ be a minimal generating set for $G$. Then $H = \langle a_i, a_j, \Phi(G) \rangle$, for $i, j \in \{1, \dots, r\}$ and $i \neq j$, is a subgroup of index $p^{d(g)-2} \geq p^i$. In particular $H$ is contained in some powerful subgroup $M$ of index $p^i$ in $G$. If $p$ is an odd prime, then $[a_i, a_j] \in \mho_{1}(M) \leq \mho_{1}(G)$. It follows that $[G,G] \leq \mho_{1}(G)$ and that $G$ is powerful. If $p=2$ then $[a_i, a_j] \in \mho_{2}(M) \leq \mho_{2}(G)$ and again it follows $G$ is powerful.
\end{proof}
This is Theorem \ref{intro theorem rank then powerful} from the introduction. \par

In the next section we focus our attention on $\mathcal{M}_1$ groups and their power structure properties. Corollary \ref{corollary i leq p-3 an M_i group is potent} and Proposition \ref{proposition mi d(G) >= i+2 then potent} will allow us to reduce to studying $2$-generator $3$-groups.

\section{On power structure of $\mathcal{M}_1$ groups}
\label{section on power structure of m1 groups}
In this section we investigate the power structure properties of $\mathcal{M}_1$ groups. As all maximal subgroups are powerful, and hence have a regular power structure by Theorem \ref{theorem powerful p-groups regular power structure}, it seems reasonable to expect a $\mathcal{M}_1$ group $G$ to have good power structure properties. 

We recall the definition of regular power structure from the introduction.
 
\begin{definition*} 
A finite $p$-group $G$ has a \textit{regular power structure} if the following three conditions hold for all positive integers $i$:
\begin{align}
 \mho_{i}(G) &=\{g^{p^{i}} \mid g \in G \}. \tag{\ref{eqn pth power}} \\
      \Omega_{i}(G) &= \{g\in G \mid o(g) \leq p^{i} \} \tag{\ref{eqn omega conditon}}.  \\
 |G:\mho_{i}(G)| &= | \Omega_{i}(G)|.  \tag{\ref{eqn index condition}}
\end{align}
\end{definition*}
\vspace{2mm}

It is our goal in this section to show that all $\mathcal{M}_1$ groups satisfy  conditions (\ref{eqn pth power}) and (\ref{eqn omega conditon}). Remarkably we show that they also satisfy (\ref{eqn index condition}), except for one specific exception -- a $3$-group of maximal class and order $3^4$.  \par

To begin, we use the results from Section \ref{subsection mi groups and potency} to show that we can reduce the problem of whether $\mathcal{M}_1$ groups have a regular power structure, to the study of $2$-generator $3$-groups. \par

\textbf{Reduction to $p=3$}: By Corollary \ref{corollary i leq p-3 an M_i group is potent}, for primes $p\geq 5$ we have that an $\mathcal{M}_1$ group $G$ is a potent $p$-group, and by Theorem \ref{theorem potent p groups regular power structure},  such groups have a regular power structure. Thus we need only focus on $p=3$. \par
\textbf{Reduction to $2$-generator groups:} Furthermore by Proposition \ref{proposition mi d(G) >= i+2 then potent} if the number of generators of an $\mathcal{M}_1$ $3$-group were greater than $2$, then the group is powerful and also has a regular power structure by Theorem \ref{theorem powerful p-groups regular power structure}. Thus in what follows we only need to consider the case of $2$-generator $3$-groups. However we note that many of our arguments work for any odd prime $p$ and so may be of independent interest. We will not need to rely on this reduction until Section \ref{section index property p=3}.

\subsection{}
Our first goal is to show that condition \ref{eqn omega conditon} is satisfied. That is, we want to show that the product of two elements of order at most $p^i$, has order at most $p^i$. 

Note that for an $\mathcal{M}_1$ group the maximal subgroups are powerful, and so have a regular power structure and in particular satisfy condition \ref{eqn omega conditon}. This condition is inherited by subgroups and so any proper subgroup of $G$ must satisfy this condition. We use this fact freely in what follows to conclude that proper subgroups which are generated by elements of exponent $p$ must have exponent $p$.

\begin{proposition}
\label{proposition m1 group 2 gen both order p}
Let $G=\langle a, b \rangle $ be an $\mathcal{M}_1$ group, with $o(a)=p, o(b)=p$, then $\exp G = p$.
\end{proposition}
\begin{proof}
First note that we can assume that $Z(G) \leq \Phi(G)$, or else the group is abelian. Next, we observe that $a, a^b \in a^G \leq M$, for some maximal subgroup $M$. As $M$ is powerful, it satisfies condition \ref{eqn omega conditon}. As $[a,b]=a \cdot a^b$, and both these factors have order $p$, it follows that $o([a,b]) \leq p$. All conjugates of $[a,b]$ in $G$ must have order at most $p$, and are contained in $M$. Then the products formed by these conjugates will have order at most $p$, and so it follows that $\exp ([G,G]) \leq p$. 

Note that for any minimal normal subgroup $Z \leq \Phi(G)$, we have that $G/Z$ satisfies the hypothesis, and so $G/Z$  must be of exponent at most $p$. Then as $Z$ was a minimal normal subgroup, so has order $p$, we may assume that the exponent of $G$ is at most $p^2$. This then means that $o(g^p)\leq p$ for all $g \in G$, and since all elements of this form are contained in a powerful, maximal subgroup $M$, it follows that the elements they generate  must have order at most $p$ and so $\exp (\mho_{1}(G)) = p$. 

Thus the product $\Phi(G)=[G,G] \mho_{1}(G)$ must also have exponent $p$, for the same reason (as $[G,G]$ and $\mho_{1}(G)$ are both contained in a powerful maximal subgroup, and so condition \ref{eqn omega conditon} holds). Then the maximal subgroups $M_{1}= \langle a, \Phi(G) \rangle$ and $M_2 = \langle b, \Phi(G) \rangle $ are both of exponent $p$, and as they are powerful by hypothesis, they must be abelian since $[M_i, M_i] \leq \mho_{1}(M_i)=1$ for $i \in \{1,2 \}$. 

Then by Fitting's Lemma \ref{lemma fitting}, the nilpotency class of $G=M_1 M_2$ is at most $2$. In particular the group $G= \langle a, b \rangle $ has exponent $p$  (as $p$ is odd, the group must be regular).
\end{proof}

\begin{corollary}
\label{cor m1 grp omega1 has exp p}
Let $G$ be a $\mathcal{M}_1$ group, then $\exp \Omega_{1}(G) \leq p$.
\end{corollary}
\begin{proof}
Notice that if $G$ were a counterexample to this corollary, then $G$ would need to contain two elements $a, b$ of order $p$ whose product had order greater than $p$. Then the elements  $a$ and $b$ cannot be contained in the same maximal subgroup so we must have $G= \langle a, b \rangle$, however now by Proposition \ref{proposition m1 group 2 gen both order p} it follows that the exponent of $G$ is $p$. Hence there can be no counterexample.
\end{proof}

We can now prove the more general result, showing that $\mathcal{M}_1$ groups satisfy condition \ref{eqn omega conditon}, that $\exp (\Omega_{i}(G)) \leq p^i$ for any positive integer $i$.

\begin{theorem}
\label{theorem omega condition holds}
Let $G$ be an $\mathcal{M}_1$ group, then $\exp \Omega_{i}(G) \leq p^i$.
\end{theorem}

\begin{proof}
The case $i=1$ is Corollary \ref{cor m1 grp omega1 has exp p}, thus we may assume that $i>1$. 

Let $N = \Omega_{1}(\Phi(G))$. The group $\bar{G}=G/N$ is an $\mathcal{M}_1$ group. Then $\Omega_{i-1}(\bar{G})=H/N$ for some subgroup $H$ of $G$ with $N \leq H$. By induction we know that $\exp \Omega_{i-1}(\bar{G}) \leq p^{i-1}$. 
Observe that if $g \in G$ has order $p^j \leq p^i$ with $j>1$, then $\bar{g} \in \Omega_{i-1}(\bar{G})$. This is because $g^{p^{j-1}}$ is a $p$th power, and is of order $p$, and so $g^{p^{j-1}} \in N$. 

Hence $\Omega_{i}(G)\leq H$, and for any $h \in H$, we have that $h^{p^{i-1}} \in N$, and the exponent of $N$ is $p$ since $N \leq \Omega_{1}(M)$ for a maximal (and hence powerful) subgroup $M$. Thus $ \exp \Omega_{i}(G) \leq \exp H \leq  p^i$.
\end{proof}

\subsection{} In this section we move to showing that condition \ref{eqn pth power} is satisfied. That is, we wish to show that for a $\mathcal{M}_1$ group, the product of two $p^i$th powers is equal to a $p^i$th power. The method of proof in this part is interesting, as we apply a result of Mann from \cite{mannpowstructure1}. 

We recall some definitions from \cite{mannpowstructure1}:
\begin{definition}
A $p$-group $G$ is a $P_1$ group if $G$, as well as \textit{all sections of $G$}, satisfy condition \ref{eqn pth power}. A $p$-group $G$ is a $P_2$ group if $G$, as well as \textit{all sections of $G$}, satisfy condition \ref{eqn omega conditon}.
\end{definition}

For example if $G$ is a $p$-group (for $p$ odd) of nilpotency class $2$, then $G$ is a regular $p$-group and so satisfies all three power structure conditions. Every section of $G$ will have nilpotency class at most $2$, and so will also satisfy all three power structure conditions. Thus in particular $G$ is a $P_2$ group. We use this fact in the next proof. 

As in the previous section, we begin with a reduced $2$-generator case. 
\begin{proposition}
\label{prop 2 gen m1 group exp p2 is p_2 group}
Let $G$ be an $\mathcal{M}_1$ group with $d(G) \leq 2$ and $\exp G \leq p^2$. Then $G$ is a $P_2$ group.
\end{proposition}
\begin{proof}
The proof is by induction on the order of the group. For orders $p$ and $p^2$ the groups are abelian and so the claim holds.

For any subgroup $Z \leq Z(G)$ we have that $G/Z$ satisfies the hypothesis. 

For any normal subgroup $N$ of $G$, we have a nontrivial $z \in N \cap Z(G)$. Then by the isomorphism theorems we know that $G/N \cong \frac{G/\langle z \rangle} {N/ \langle z \rangle}$, and notice that $G/ \langle z \rangle$ is a $\mathcal{P}_2$ group by induction, hence $G/N$ is a $P_2$ group. Hence all proper quotients of $G$ are $P_2$ groups.  We also know that $G$ itself satisfies condition (\ref{eqn omega conditon}) by Theorem \ref{theorem omega condition holds}. Thus all that remains is to prove that proper subgroups of $G$ are $P_2$. 

For any maximal subgroup $M$ of $G$, we have that $[M,M,M] \leq [\mho_{1}(M),M] \leq \mho_{2}(M)=1$, since $M$ is powerful. Then in particular all proper subgroups have class at most $2$. For $p$ odd, groups of class $2$ are $P_2$ groups.

Thus we have established that $G$ is a $P_2$ group.
\end{proof}

We now quote a result from \cite{mannpowstructure1}.
\begin{theorem}[{\cite[Corollary 4]{mannpowstructure1}}] \label{theorem p2 group is p1 group} A $P_2$ group is a $P_1$ group.

\end{theorem}
Combining Proposition \ref{prop 2 gen m1 group exp p2 is p_2 group} and Theorem \ref{theorem p2 group is p1 group}  we obtain the following corollary.

\begin{corollary}
\label{cor  2 gen m1 group exp p2 has set pth powers}
Let $G$ be an $\mathcal{M}_1$ group with $d(G) \leq 2$ and $\exp G \leq p^2$. Then $\mho_{1}(G) = \{ g^p \mid g \in G \}.$
\end{corollary}
The results so far can now be used to prove that for an $\mathcal{M}_1$ group, the product of $p$th powers is equal to a $p$th power.
\begin{proposition}
\label{prop m1 group pth powers equal set pth powers}
Let $G$ be an $\mathcal{M}_1$ group, then $\mho_{1}(G)= \{ g^p \mid g \in G \}$.
\end{proposition}
\begin{proof}
As usual, we may assume that $d(G)=2$, by Proposition \ref{proposition mi d(G) >= i+2 then potent}. Next we show that we may assume that the exponent of the group is at most $p^2$. 

If $\exp G > p^2$, then $\mho_1( \mho_1 (G)) \neq 1$. Then $\mho_1 (\mho_1 (G))$ contains some minimal normal subgroup $Z$. Then by induction, it follows that for any $a, b \in G$ we have that $a^p b^p = c^p z$ for some $c \in G$ and $z \in Z$. The element $c$ is contained in some maximal subgroup $M$. Notice that $z \in \mho_1( \mho_1 (G)) \leq \mho_{1}(M)$, and so $z=m^p$ for some $m \in M$  (here we use that $M$ is powerful and so satisfies condition \ref{eqn pth power}).

But now, $a^p b^p = c^p z = c^p m^p = x^p$ for some $x \in M$, where we have used again that $c, m \in M$ and $M$ is powerful. Thus it follows that $\mho_{1}(G) = \{ g^p \mid g \in G \}$. 

Hence we may assume that the exponent of $G$ is at most $p^2$. This is now the case of Corollary \ref{cor  2 gen m1 group exp p2 has set pth powers} above and the result follows.
\end{proof}

We next prove that for an $\mathcal{M}_1$ group, the first Agemo subgroup, $\mho_{1}(G)$ is powerful. This will help us generalise Proposition \ref{prop m1 group pth powers equal set pth powers} and will also be used in the proof Proposition \ref{prop general index result 3grps} . We also note that this is interesting in it's own right, in light of a question of Wilson, see \cite[Question 4.8]{wilsonpowerstrucpowpgroups}.
\begin{question}[{\cite[Question 4.8]{wilsonpowerstrucpowpgroups}}]
If $G$ has the property that $\mho_{k}(\mho_{k}(G))$ is the set of $p^{k}$th powers of elements of $G$ for all $k$, then must $\mho_{1}(G)$ be powerful.
\end{question}
This property has been verified for many families of groups, and we are pleased to be able to add $\mathcal{M}_1$ groups to this list of families.

\begin{proposition}
\label{prop G m1 then agemo 1 is powerful}
Let $G$ be an $\mathcal{M}_1$ group, then $\mho_1 (G)$ is powerful.
\end{proposition}
\begin{proof}
First, observe that by Proposition \ref{prop m1 group pth powers equal set pth powers} we know that $\mho_1 (G) = \{ g^p \mid g \in G \}$. Thus $\mho_{1}(\mho_{1}(G))= \langle g^{p^{2}} \mid g \in G\rangle = \mho_{2}(G)$. We wish to show that $\mho_1(G)$ is powerful. To do this, we may quotient by $\mho_{2}(G)=\mho_{1}(\mho_{1}(G))$ and show that the resulting group $H=G/\mho_{2}(G)$, which is of exponent at most $p^2$, has $\mho_{1}(H)$ being abelian. (Recall that the Agemo subgroups behave well under taking quotients so $\mho_{1}(\frac{G}{\mho_{2}(G)}) = \frac{\mho_{1}(G)}{\mho_{2}(G)}$).

Let $a,b \in \mho_{1}(H)$; as $H$ is still an $\mathcal{M}_1$ group, then by Proposition \ref{prop m1 group pth powers equal set pth powers} we can write $a=x^p$ and $b=y^p$ for some $x,y \in H$. Let $M$ be any maximal subgroup of $H$ which contains $x$. Note that $b$ is contained in all maximal subgroups of $H$ (as it is a $p$th power). We observe that the nilpotency class of $M$ is at most $2$, since $[M,M,M] \leq [\mho_{1}(M), M] \leq M^{p^{2}}=1$. Then as $x$ and $b$ are both contained in $M$, we have that $[a,b]=[x^p,b]=[x,b]^p \in \mho_{1}(\mho_{1}(H))=1$. In particular $\mho_1(H)$ is abelian. The result now follows.

\end{proof}

We can now prove that for $\mathcal{M}_1$ groups, the product of any two $p^i$th powers is always a $p^i$th power. Indeed, by Proposition \ref{prop m1 group pth powers equal set pth powers} we know that $\mho_1(G)= \{ g^p \mid g \in G \}$. Then for $i>1$, we have that 
$$a^{p^{i}}b^{p^{i}} = (a^p)^{p^{i-1}}(b^p)^{p^{i-1}}=c^{p^{i-1}}$$ for some $c \in \mho_1(G)$, since $\mho_{1}(G)$ is powerful by Proposition \ref{prop G m1 then agemo 1 is powerful}. However then $c=d^p$ for some $d \in G$ by Proposition \ref{prop m1 group pth powers equal set pth powers}. Hence $a^{p^{i}}b^{p^{i}}=d^{p^{i}}$.
Thus we have proved the following theorem.
\begin{theorem}
\label{theorem agemo i is p i th powers}
Let $G$ be an $\mathcal{M}_1$ group. Then $\mho_i(G) = \{ g^{p^{i}} \mid g \in G \}$ for $i \geq 1$.
\end{theorem}

\subsection{} 
\label{section index property p=3}
We now address the remaining power structure property, condition \ref{eqn index condition}. We recall the condition below. 
\begin{align}
 |G:\mho_{i}(G)| &= | \Omega_{i}(G)|.  \tag{\ref{eqn index condition}}
\end{align}
This condition may seem the least natural of the three. We offer here one possible way to motivate this condition: For an abelian $p$-group $G$, the $p^i$th power map on $G$ is a homomorphism with image $\mho_{i}(G)$ and kernel $\Omega_{i}(G)$, and thus condition \ref{eqn index condition} follows by the isomorphism theorems. For $p$-groups in general (and even for groups with a regular power structure) this map need not be a homomorphism, but perhaps this property could be thought of as trying to capture some essence of that behaviour. \par

We believe that status of condition \ref{eqn index condition} for $\mathcal{M}_1$ groups is surprising. The first two properties have now been established for $\mathcal{M}_1$ groups and any odd prime $p$. Furthermore by Proposition \ref{proposition mi d(G) >= i+2 then potent} we know that the final condition holds for primes $p \geq 5$, and by Proposition \ref{proposition mi d(G) >= i+2 then potent} all regular power structure properties hold for $p=3$ if $d(G)>2$. Thus given all the cases where condition \ref{eqn index condition} holds, it is surprising that in the remaining case of $\mathcal{M}_1$ $3$-groups $G$ with $d(G)=2$, we will see that the condition \ref{eqn index condition} need not hold. Even more intriguing is that there is only one example where this conditions fails - and in this case the group is of a very small order, $|G|=3^4$. 

\begin{example}
\label{example smallgroup 81 10}
The following group can be easily constructed in GAP as \texttt{SmallGroup(81,10)}. For completeness we list below a power commutator presentation for the group.
$$J = \langle a_1, a_2, a_3, a_4 \mid a_1^3 = a_4, a_2^3 = (a_4)^2, a_3^3 =1, a_4^3=1, [a_2,a_1]=a_3, [a_3, a_1]=a_4   \rangle,$$ the four remaining commutator relations not listed are trivial.  

We state some properties of this group - these can be readily verified in GAP, or with more effort, by hand.\par

This group has nilpotency class $3$ and order $3^4$ and so is a group of maximal class. It has four maximal subgroups, one of which is abelian (isomorphic to $C_3 \times C_9$), and the other three all isomorphic to $$\langle x, y \mid x^3, y^9, [x,y]=y^3 \rangle,$$ a semidirect product of the form $C_9 \rtimes C_3$. 

Then it is clear that our group is an $\mathcal{M}_1$ group. 

In this group we have that $\mho_{1}(J) = \langle a_4 \rangle $ and $\Omega_{1}(J) = \langle a_3, a_4 \rangle$. Hence we have that $|\mho_1 (J)| |\Omega_{1}(J)| = 27 \neq |J| = 81$, and so the group does not satisfy condition \ref{eqn index condition}. We also remark here that this is the only $\mathcal{M}_1$ group of order $3^4$ and of maximal class (nilpotency class $3$), which can be readily verified in GAP.
\end{example}

For $p$ odd, this is the only example of an $\mathcal{M}_1$ $p$-group which does satisfy condition \ref{eqn index condition}. Proving this, and the method we use to do so, is one of the main contributions of the paper.

By the reductions outlined at the start of Section \ref{section on power structure of m1 groups}, we only need to be concerned with $2$-generator $3$-groups. Unlike previous sections where we still offered arguments that worked for any odd prime, here we will take advantage of these reductions and focus only on $p=3$. In contrast to the very theoretical arguments in the previous section, we now move to more concrete arguments making use of the classification of $3$-groups of small order.

As all groups of exponent $3$ are regular, it is clear that the smallest exponent for which a counterexample to condition \ref{eqn index condition} can occur is exponent $3^2$. We exhibited such a group in Example \ref{example smallgroup 81 10}. To begin, we show that it is the only counterexample of exponent $3^2$. 

\begin{lemma}
\label{lemma G m1 group exp p2 class at most 3}
Let $G$ be an $\mathcal{M}_1$ group of exponent $p^2$. Then the nilpotency class of $G$ is at most $3$.
\end{lemma}
\begin{proof}
Every maximal subgroup of $G$ is powerful and of exponent at most $p^2$. For a powerful $p$-group $H$ we have that $[H,H,H] \leq H^{p^{2}}$ and so it follows that every proper subgroup of $G$ must have nilpotency class at most $2$. Now, by a remarkable theorem of MacDonald \cite[Corollary 1 to Theorem 1]{Macdonaldclassicaltheoremonnilpotentgroups}, it is known that if all proper subgroups of $G$ have class at most $2$ then $G$ has class at most $3$.
\end{proof}

\begin{lemma}
\label{lemma only one coutner example exp 9}
There is exactly one $2$-generator, $\mathcal{M}_1$ $3$-group of exponent $9$ which does not satisfy $|G|=|\mho_1 (G)| | \Omega_{1}(G)|$.
\end{lemma}
\begin{proof}
Let $G= \langle a, b \rangle $ be a $2$-generator $\mathcal{M}_1$ group of exponent $9$. By Lemma \ref{lemma G m1 group exp p2 class at most 3} the class of $G$ is at most $3$. The largest (most free) $2$-generator $3$-group of exponent at most $9$ and nilpotency class at most $3$ has order $3^8$. This group can be constructed using the nilpotent quotient algorithm from the GAP package "nq" \cite{nq2.5.4}. This group is not a $\mathcal{M}_1$ group, but any $2$-generator $\mathcal{M}_1$ group with exponent at most $9$ and class at most $3$ will be a quotient of this group. In particular their order will be at most $3^7$.

We have now reduced to a finite problem, and as all the $3$-groups of orders $3^n$, $n \in \{1, \dots, 7 \}$ are classified we can check the classification and see that there is indeed only one counterexample. The group of Example \ref{example smallgroup 81 10}.
\end{proof}

In what follows we use the well known commutator expansion formula (see \cite{mckay2000finite}, Exercise 1.2). If $G$ is a group, $x,y \in G,$ and $n \in \mathbb{N}$ then  
\begin{equation}
\label{p hall collection  (xy)^p^n}
(xy)^{p^{n}} \equiv x^{p^{n}}y^{p^{n}} \left(\thinspace \text{mod} \thinspace \mho_{n}(\gamma_{2}(T)) \mho_{n-1}( \gamma_{p}(T))\dots \gamma_{p^{n}}(T)\right)
\end{equation} where $T= \langle x,y \rangle$. 

\begin{lemma}
\label{lemma mho1 G is either mho1M or 3 times this}
Let $G$ be a $\mathcal{M}_1$ $3$-group and let $M$ be a maximal subgroup of $G$. Then either $|\mho_1(G)| = | \mho_1(M)|$ or $|\mho_1(G)| =3 \cdot |\mho_1 (M) | $.
\end{lemma}
\begin{proof}
We may write $G= \langle M, a \rangle$ for some element $a \in G$. Then any element of the group can be written in the form $ m a^i$ for some $m \in M$ and $i \in \{ 0 ,1, 2 \}$ . 

Now we know the form of a generic element, we can consider the form of a generic $3$rd power. 

Using formula \ref{p hall collection  (xy)^p^n} above, we obtain:

$$ 
(m a^i)^3 =(m)^3 (a^i)^3  \mho_{1}(\gamma_{2}(T)) \gamma_{3}(T),
$$
where $T= \langle m, a^i \rangle$.

Notice that $\gamma_{2}(T)^3 \leq \mho_{1}(M)$. \par
We next consider  $\gamma_{3}(T) = \gamma_3(\langle m, a^i \rangle)$.   Notice that for some collection of elements $m_i \in M$ there is a maximal subgroup $N = \langle a, m_1, \dots, m_n \rangle$ in $G$. As it is maximal, it contains $\Phi(G)$ and so $\gamma_{2}(T) \leq N$. Then $[a, m, a] \in N^3 = \langle a^3, m_1^3, \dots, m_n^3$ by Proposition \ref{proposition G= a1,..,ar then G^p = a_1^p, a_r^p}. Also notice that $[a,m,m] \in M^3$. In particular it follows that $\gamma_{3}(T) \leq \langle a^3, \mho_{1}(M) \rangle$. Then it follows that $(m a^i)^3 = \hat{m}^3 a^j$ for some $m \in M$ and $j \in {0,1,2}$ (notice $a^3 \in M$). 

Then we see we have $3$ choices for $j$, and $\hat{m}^3$ can take $|\mho_{1}(M)|$ values (where we note that $\mho_{1}(M)$ consists solely of $3$rd powers because it is powerful). Then there are at most $3 \cdot |\mho_{1}(M)|$ possible $3$rd powers in $G$. 

As we proved in Theorem \ref{theorem agemo i is p i th powers}, $\mho_{1}(G)$ consists solely of $3$rd powers and so $|\mho_{1}(G)| \leq 3 \cdot |\mho_{1}(M)|$. 

Clearly $\mho_{1}(G) \geq \mho_{1}(M)$ and so either $|\mho_1(G)| = | \mho_1(M)|$ or $|\mho_1(G)| =3 \cdot |\mho_1 (M) | $ as required. 

\end{proof}

\begin{lemma}
\label{lemma deal with case mho1 G is bigger than mho1 M}
Let $G$ be an $\mathcal{M}_1$ $3$-group of exponent at least $p^2$, and let $M$ be a maximal subgroup of $G$ such that $\Omega_{1}(G) \leq M$. If $\mho_1(G) \neq \mho_1(M)$ then $|G|= |\mho_{1}(G)| |\Omega_{1}(G)|$.
\end{lemma}
\begin{proof}
$\Omega_{1}(G)$ consists of all elements of order at most $p$ in $G$ by Theorem \ref{theorem omega condition holds}. Similarly, $\Omega_{1}(M)$  consists of all elements of order $p$ in $M$, because $M$ is powerful. Then because $\Omega_{1}(G) \leq M$ we must have that 
\begin{equation}
\label{eqn omega1G = omega!M}
    \Omega_{1}(G) = \Omega_{1}(M). 
\end{equation}

Let $|G|=p^n$, then for the maximal subgroup $M$ we have $|M|=p^{n-1}$. As $M$ is powerful, we have that condition \ref{eqn index condition} holds for $M$:
\begin{equation}
\label{eqn omega 3 holds in proof}
p^{n-1} = |M| = |\mho_{1}(M)| |\Omega_{1}(M)|.
\end{equation}

By hypothesis  $\mho_{1}(G) \neq \mho_{1}(M)$ and so we have $|\mho_{1}(G)| = p|\mho_{1}(M)|$ by Lemma \ref{lemma mho1 G is either mho1M or 3 times this}. Hence 
\begin{align*}
| \mho_1 (G)| | \Omega_{1}(G)| &= p | \mho_1 (M)| | \Omega_{1}(G)| \tag{\text{Lemma \ref{lemma mho1 G is either mho1M or 3 times this}}}\\
                                &= p | \mho_1 (M)| | \Omega_{1}(M)| \tag{\text{by \ref{eqn omega1G = omega!M}}}\\
                               &= p \cdot p^{n-1} \tag{\text{by \ref{eqn omega 3 holds in proof}}}\\
                                &= |G|.
\end{align*}

\end{proof}

We have seen that there is only one counterexample of  exponent  at most $9$. We now wish to show there are no counterexamples of exponent $3^i$ with $i \geq 3$. In this case, by Theorem \ref{theorem omega condition holds} we know that $\exp \Omega_{1}(G) = 3$ and so we must have that $\Omega_{1}(G)$ is contained in a maximal subgroup $M$ of $G$. Then by Lemma \ref{lemma deal with case mho1 G is bigger than mho1 M} in the cases that follow we assume we have that $\mho_{1}(G) = \mho_{1}(M)$.

The point of the following lemma is to say if a counterexample $G$ exists, then we can bound the size of $\Omega_{1}(G)$. This will allow us to reduce our search for counterexamples to checking a finite number of groups later on.

\begin{lemma}
\label{lemma omega 1 has size at most 9}
Let $G$ be a $\mathcal{M}_1$, $2$-generator $3$-group, such that $\Omega_{1}(G) \leq M$ for some maximal subgroup $M$ and additionally that $|\mho_{1}(G)| = |\mho_{1}(M)|$. Then $|\Omega_{1}(G)| \leq 3^2$.
\end{lemma}
\begin{proof}
As $M$ is powerful we have $|\mho_1(M)| |\Omega_1 (M)| = |M|$. As $\Omega_1(G) \leq M$ we have $\Omega_{1}(G) = \Omega_{1}(M)$, and by the hypothesis we have $|\mho_{1}(G)| = |\mho_{1}(M)|$. Hence
$$ 
|\mho_1(G)| |\Omega_1 (G)|=|\mho_1(M)| |\Omega_1 (M)| = |M| = |G|/3.
$$
Rearranging yields:
$$3 \cdot |\Omega_{1}(G)| = |G|/|\mho_1(G)| = |G/\mho_{1}(G)|.$$
The group $G/ \mho_{1}(G)$ is $2$ generator of exponent $3$. It is well known the largest such $3$-group has order $3^3$. 
Hence $3 \cdot |\Omega_{1}(G)| \leq 3^3$ and so $|\Omega_{1}(G)| \leq 3^2$ as required.
\end{proof}

In the proof of the following lemma we use a result due to L. Wilson from \cite{wilsonpowerstrucpowpgroups}. We recall some notation from \cite{wilsonpowerstrucpowpgroups}. 
 
We let $\mathcal{O}_p$ be the class of all $p$-groups for which $\Omega_{k}(G)$ is the set of elements of order dividing $p^k$ for all $k$. Theorem \ref{theorem omega condition holds} tells us that for odd primes $p$, $\mathcal{M}_1$ $p$-groups are in $\mathcal{O}_p$.

\begin{lemma}[{\cite[Lemma 2.1]{wilsonpowerstrucpowpgroups}}]
\label{lemma wilson omegas order}
Let $G$ be in $\mathcal{O}_p$. Then for all $m$ and $k$ 
$$\Omega_{k}(G/\Omega_{m}(G))=\Omega_{m+k}(G)/\Omega_{m}(G).$$
\end{lemma}

\begin{lemma}
\label{lemma if G/omega satisfies then so does G}
If $G$ is a $\mathcal{M}_1$, $2$-generator, $3$-group  of exponent at least $3^3$ and $H=G/ \Omega_{1}(G)$ satisfies $|H|= |\Omega_{1}(H)| | \mho_{1}(H)|$, then $|G|= |\Omega_{1}(G)| | \mho_{1}(G)|$.
\end{lemma}
\begin{proof}
As the exponent of $G$ is at least $3^3$, and the group satisfies condition \ref{eqn omega conditon} by Theorem \ref{theorem omega condition holds}, then $ \exp \Omega_2(G) \leq 9$ and so $\Omega_{1}(G)$ is a proper subgroup of $G$ and thus is contained in some maximal subgroup $M$. Then clearly $\Omega_{1}(G)$ is also contained in $M$ and so by Lemma \ref{lemma deal with case mho1 G is bigger than mho1 M} we only need to consider the case where 
\begin{equation}
\label{eqn mhog = mho1 in proof of lemma}
    \mho_{1}(G)= \mho_{1}(M).
\end{equation}
We will show that this situation never occurs. 

Let $|G|=3^n$ and $|\Omega_{1}(G)|=3^a$, clearly $a>0$. Note that $\Omega_{1}(M)=\Omega_{1}(G)$ and $\Omega_{2}(M) = \Omega_{2}(G)$. 

We now wish to consider what happens when we quotient by $\Omega_{1}(G)$, we shall denote the image under this quotient with bar notation.

$$p^{n-1-a}=|\bar{M}| = | \Omega_{1}(\bar{M})| \cdot | \mho_{1}(\bar{M})|$$ 
since $\bar{M}$ is powerful and so $\bar{M}$ satisfies \ref{eqn index condition}.

By hypothesis we know that $|\bar{G}|=| \Omega_{1}(\bar{G})| | \mho_{1}(\bar{G})|$. We also know that $|\bar{G}|=|G|/|\Omega_{1}(G)| = 3^{n-a}$. Hence
\begin{equation}
    \label{eqn omega G bar agemo G bar = p^n-a}
    | \Omega_{1}(\bar{G})| | \mho_{1}(\bar{G})|=3^{n-a}.
\end{equation}

However we also have that
\begin{equation*}
    \Omega_{1}(\bar{M}) = \Omega_{1}(\frac{M}{\Omega_{1}(M)}) \underset{\text{Lemma \ref{lemma wilson omegas order}}}{=} \frac{\Omega_{2}(M)}{\Omega_{1}(M)} = \frac{\Omega_{2}(G)}{\Omega_{1}(G)}.
\end{equation*}

and also
\begin{equation*}
    \Omega_{1}(\bar{G}) = \Omega_{1}(\frac{G}{\Omega_{1}(G)}) \underset{{\text{Lemma \ref{lemma wilson omegas order}}}}{=} \frac{\Omega_{2}(G)}{\Omega_{1}(G)}. 
\end{equation*}
Hence 
\begin{equation}
\label{equation omega1barm = omega1 bar G in proof}
   \Omega_{1}(\bar{M})=\Omega_{1}(\bar{G}).
\end{equation}

Also 
\begin{equation}
\label{equation in proof bar M^p=bar G^p} 
 \mho_{1}(\bar{M})=\mho_{1}(\frac{M}{\Omega_{1}(G)})= \frac{\mho_{1}(M)}{\Omega_{1}(G)}= \frac{\mho_{1}(G)}{\Omega_{1}(G)} = \mho_{1} (\bar{G}). 
\end{equation}

Hence

\begin{align*} 3^{n-1-a}  &= |\Omega_{1}( \bar{M})| | \mho_{1}(\bar{M}) | \\
                         &= |\Omega_{1}(\bar{G})||\mho_{1}(\bar{M})| &&\text{(by \ref{equation omega1barm = omega1 bar G in proof})} \\
                         &= | \Omega_{1}(\bar{G})||\mho_{1}(\bar{G})| &&\text{(by \ref{equation in proof bar M^p=bar G^p})} \\
                         &=3^{n-a} &&\text{(by \ref{eqn omega G bar agemo G bar = p^n-a})}.
\end{align*}
Thus we obtain a contradiction, and so the case when $\mho_{1}(G)=\mho_{1}(H)$ cannot occur. The proof is complete.

\end{proof}

We can now prove that there are no counterexamples of exponent $p^3$. 

\begin{lemma}
\label{lemma no counter examples of exponent p^3}
There are no $\mathcal{M}_1$ $2$-generator $3$-groups of exponent $27$ such that $|G| \neq |\Omega_{1}(G)| |\mho_1 (G)|$.
\end{lemma}
\begin{proof}
Suppose that $G$ is a minimal counterexample with respect to its order. By Lemma \ref{lemma if G/omega satisfies then so does G} we must have that $H=G/\Omega_{1}(G)$ is such that $|H| \neq |\Omega_{1}(H)||\mho_{1}(H)|$. Notice that $H$ is $2$-generator $3$-group with $|H| \neq |\Omega_{1}(H)||\mho_1(H)|$ and smaller order than $G$. As $G$ is the smallest counterexample of exponent $p^3$, the group $H$ must have smaller exponent to avoid contradicting this minimality.Thus $H$ must have exponent $9$ (since $\exp \Omega_1 (G) = 3$).

By Lemma \ref{lemma only one coutner example exp 9} there is only one possibility for $H$, and we see that $|H|=3^4=81$. By Lemma \ref{lemma deal with case mho1 G is bigger than mho1 M} we can assume that $|\mho_{1}(G)|=|\mho_{1}(M)|$ for some maximal subgroup $M$, with $\Omega_{1}(G) \leq M$. Now by Lemma \ref{lemma omega 1 has size at most 9} we can assume that $|\Omega_{1}(G)|\leq 9$. Then if $G/\Omega_{1}(G) \cong H$ we must have that $|G| \leq 3^2 \cdot 3^4$. 

By checking the classification of the $3$-groups of order at most $3^6$, for example by using GAP, we see that all $\mathcal{M}_1$ groups $G$ of exponent $27$ and $|G| \leq 3^6$ satisfy $|G|= |\mho_{1}(G)||\Omega_{1}(G)|$. 
\end{proof}

\begin{proposition}
\label{prop no counter examples exp p^3}
Let $G$ be a $\mathcal{M}_1$ $3$-group with $2$-generators and exponent at least $3^3$. Then $|G|=|\mho_{1}(G)||\Omega_{1}(G)|$.
\end{proposition}
\begin{proof}
Let the exponent of the group be $3^e$, with $e\geq 3$. The proof is by induction on e. The base case $e=3$ is established by  Lemma \ref{lemma no counter examples of exponent p^3}. Now consider some $e>3$ and suppose the claim holds for smaller values.

As usual we may assume that $G$ contains a maximal subgroup $M$ containing $\Omega_{1}(G)$. The group $G/\Omega_{1}(G)$ is of exponent $p^{e-1}$, because if $g$ had order $p^e$, then $g^{p^{e-1}} \in \Omega_{1}(G)$. 

As $e-1 \geq 3$ it follows inductively that $H=G/\Omega_{1}(G)$ has the property that $|H|=|\Omega_{1}(H)||\mho_{1}(H)|$. Then by Lemma \ref{lemma if G/omega satisfies then so does G} $G$ also has the property and the proof follows by induction.
\end{proof}

We now seek to prove the more general result, condition \ref{eqn index condition}. First we prove the following lemma. 

\begin{lemma}
\label{lemma no groups such that G/omega i is counter example}
There are no  $\mathcal{M}_1$ $2$-generator $3$-groups $G$ such that $G/\Omega_{i}(G)$ is isomorphic to the group from Example \ref{example smallgroup 81 10}, for any $i \geq 1$.
\end{lemma}
\begin{proof}
We let $J$ denote the group from Example \ref{example smallgroup 81 10}. \par

The proof is by induction on $i$. First consider the base case $i=1$.
Suppose that $G$ is such a group and so $G/\Omega_{1}(G) \cong J$. Then $\exp G \geq 27$ (Lemma \ref{lemma wilson omegas order}). By Proposition \ref{prop no counter examples exp p^3} we know that $|G|=|\Omega_{1}(G)||\mho_{1}(G)|$. Then $|\Omega_{1}(G)| = |G/\mho_{1}(G)|$. As $G/\mho_{1}(G)$ is at most $2$-generator of exponent $3$, its order is at most $27$. Therefore the order of $G$ is at most $3^4 \cdot 3^3$. The groups of order $3^7$ have been classified. From this classification one can verify that there are $3$-groups $G$ of order at most $3^7$ with $G/\Omega_{1}(G) \cong J$, where $J$ is the group from Example \ref{example smallgroup 81 10}.

Now suppose $i=k$ and the claim has been established for smaller values. 
Using Lemma \ref{lemma wilson omegas order}, notice that:
\begin{equation}
    \Omega_{1}(\frac{G}{\Omega_{i-1}(G)}) = \frac{\Omega_{i}(G)}{\Omega_{i-1}(G)}.
\end{equation}
Then 
\begin{equation}
    \frac{(G/\Omega_{i-1}(G))}{\Omega_{1}((G/\Omega_{i-1}(G)))} = \frac{G/\Omega_{i-1}(G)}{\Omega_{i}(G)/\Omega_{i-1}(G)} = \frac{G}{\Omega_{i}(G)} \cong J.
\end{equation}
The result then follows by the base case, as there can be no group $G/\Omega_{i-1}(G)$.
\end{proof}

\begin{remark}
In the proof of the next proposition, our approach very closely follows the approach of Wilson in \cite{wilsonpowerstrucpowpgroups}. We note that we also followed the same method in \cite{williams2019quasipowerful} to prove that \textit{quasi-powerful} $p$-groups satisfy property \ref{eqn index condition} for odd primes $p$. We think that the approach in \cite{wilsonpowerstrucpowpgroups} is a very nice method for establishing the final power structure property.
\end{remark}

\begin{proposition}
Let $G$ be a $\mathcal{M}_1$ $2$-generator $3$-group such that $G$ is not isomorphic to the group from Example \ref{example smallgroup 81 10}. Then $|G|=|\Omega_{i}(G)||\mho_{i}(G)|$ for any $i \geq 1$.
\label{prop general index result 3grps}
\end{proposition}
\begin{proof}
We use induction on $i$. The base case $i=1$ is established by Proposition \ref{prop no counter examples exp p^3} ($\exp G \geq 27$) and Lemma \ref{lemma only one coutner example exp 9} ($\exp G = 9$) and the fact that groups of exponent $3$ are regular. \par

Thus we may assume the result holds for $i$. We wish to find the order of $\mho_{i+1}(G)$. By Theorem \ref{theorem agemo i is p i th powers}, we have that $\mho_{i+1}(G)=\mho_{i}((\mho_{1}(G)))$. As $\mho_{1}(G)$ is powerful (by Proposition \ref{prop G m1 then agemo 1 is powerful}), we  know that $\mho_{1}(G)$ satisfies condition \ref{eqn index condition}. In particular we can conclude that
\begin{equation}
\label{eqn a}
    | \mho_{i+1}(G)|=|\mho_{1}(G):\Omega_{i}(\mho_{1}(G))|.
\end{equation}

By Theorem \ref{theorem omega condition holds} we know that the exponent of $\Omega_{i}(G)$ is at most $p^i$ and so we have $\Omega_{i}(\mho_{1}(G))=\Omega_{i}(G) \cap \mho_{1}(G)$.
Then 
$$\mho_{1}(G)/\Omega_{i}(\mho_{1}(G))  \cong \mho_{1}(G) \Omega_{i}(G)/\Omega_{i}(G) = (G/\Omega_{i}(G))^p.$$
Thus 
\begin{equation}
\label{eqn b}
|\mho_{1}(G): \Omega_{i}(\mho_{1}(G))| = |(G/\Omega_{i}(G))^p|. 
\end{equation}
Since quotients of $\mathcal{M}_1$ groups are $\mathcal{M}_1$ groups, we have that $G/\Omega_{i}(G)$ is an $\mathcal{M}_1$ group. If $G/\Omega_{1}(G)$ is $2$-generator then by Lemma \ref{lemma no groups such that G/omega i is counter example} we can apply the base case, to find that
\begin{align}
\label{eqn G/omega_i(G) ^p}
    |(G/\Omega_{i}(G))^p| &= |G/\Omega_{i}(G): \Omega_{1}(G/\Omega_{i}(G))| \\
                          &=|G:\Omega_{i+1}(G)|. \tag{\text{Lemma \ref{lemma wilson omegas order}}}
\end{align}

If $G/\Omega_{i}(G)$ is cyclic then \ref{eqn G/omega_i(G) ^p} clearly holds. In either case we see that 
\begin{align*}
    |\mho_{i+1}(G)| &= |\mho_{1}(G) : \Omega_{i}(\mho_{1}(G))| \tag{by \ref{eqn a}} \\
                    &= |\mho_{1}((G/\Omega_{i}(G)))| \tag{by \ref{eqn b}} \\
                    &=| G : \Omega_{i+1}(G)| \tag{by \ref{eqn G/omega_i(G) ^p}}.
\end{align*}

\end{proof}

\subsection{}
Putting together the results of the previous sections we obtain the following Theorem. This is Theorem \ref{intro theorem main} from the Introduction.
\begin{theorem}
Let $p$ be an odd prime. Let $G$ be an $\mathcal{M}_1$ $p$-group. Then $G$ has a regular power structure, unless $G$ is isomorphic to the group of Example \ref{example smallgroup 81 10}, in which case condition \ref{eqn index condition} fails to hold.
\end{theorem}
\begin{proof}
If $p>3$ the group is potent by Corollary \ref{corollary i leq p-3 an M_i group is potent} and so has a regular power structure. If $d(G)>2$ the group is powerful by Proposition $\ref{proposition mi d(G) >= i+2 then potent}$. Thus we only need to consider $2$-generator $3$-groups. 

The second power structure property \ref{eqn omega conditon} is established by Theorem \ref{theorem omega condition holds}. The first power structure property is established by Theorem \ref{theorem agemo i is p i th powers}. 

Proposition \ref{prop general index result 3grps} establishes the third and final power structure property for all $2$-generator $\mathcal{M}_1$ $3$-groups with exactly one exception, the group from Example \ref{example smallgroup 81 10}.

\end{proof}

\section{Further Remarks}
\makeatletter
\def\@secnumfont{\bfseries}
\makeatletter
\label{section further remarks}
\numberwithin{theorem}{section}
\subsection{}\label{a note on p=2}
In the theory of $p$-groups it is well known that small primes can behave differently to large primes - often being more difficult to deal with. Newman refers to this as the tyranny of the small in \cite{newmangroupsofprimepowerorder}. 
When studying $\mathcal{M}_1$ groups our focus has been on odd primes, with a large part of the paper spent dealing with the difficult case $p=3$. We briefly address here the case $p=2$. We should not expect $p=2$ to behave as well as the other primes. It often happens that some condition which guarantees regular power structure for odd primes fails when considered for $p=2$. For example powerful $2$-groups need not have a regular power structure (see  \cite[page 2]{onelemsorderpinpowpgroups} for an example of a powerful $2$-group where condition \ref{eqn omega conditon} fails to hold). In \cite{williams2019quasipowerful} we introduced the notion of a quasi-powerful $p$-group for odd primes $p$. We showed that for odd primes these groups have a regular power structure, but if the definition was used unmodified for $p=2$, the resulting groups need not have regular power structure. 

We can readily find examples of $\mathcal{M}_1$ $2$-groups which do not have a regular power structure. For example, there is a group of order $64$ (which can be created in GAP as \texttt{SmallGroup(64,31)}) which fails all three power structure properties. It would be interesting to explore if the theory in this paper could be adapted for $2$-groups.

\subsection{}
A natural generalisation  to consider is what happens if  we replace powerful with potent. For $p=3$ the notions are the same, and so we focus now on the case $p \geq 5$. We could then ask for $p \geq 5$ if all maximal subgroups of $G$ are potent, must $G$ satisfy conditions \ref{eqn pth power} or \ref{eqn omega conditon} or \ref{eqn index condition}. 

Using GAP, we can find examples of $5$-groups with all maximal subgroups potent, such that condition \ref{eqn omega conditon} fails, and also examples where condition \ref{eqn index condition} fails (for example \texttt{SmallGroup(78125, 784)} fails both). However in all the examples we have found condition \ref{eqn pth power} does hold. Thus we ask the following question.
\begin{question}
Does there exist a $p$-group ($p \geq 5$) such that all maximal subgroups are potent, for which condition \ref{eqn pth power} fails to hold.
\end{question}

\subsection{}
\label{further remark on what happens to pow structure if i is increased}
We may also ask, how are the power structure conditions affected if $i$ is increased from $1$. Corollary \ref{corollary i leq p-3 an M_i group is potent} gives us a result in this direction, for instance for $p \geq 5$ we have that if all subgroups of index $p^2$ are powerful then $G$ is potent. We might then ask how sharp is this result, for instance what can be said when $p=3$ about $p$-groups with all subgroups of index $p^2$ powerful? We have examples of $\mathcal{M}_2$ $3$-groups to demonstrate each of the properties can fail. For instance the groups with the following SmallGroup Ids (in GAP or MAGMA) all fail condition \ref{eqn pth power} and condition \ref{eqn index condition}.\par
{
\footnotesize
\texttt{
 [ 2187,83 ], [ 2187, 84 ], [ 2187, 85 ], [ 2187, 90 ], [ 2187, 91 ], [ 2187, 92 ]}
 }

\texttt{SmallGroup(81,7)} is an example of an $\mathcal{M}_2$ $3$-group for which condition \ref{eqn omega conditon} fails; interestingly it is the only example we could find.

\bibliographystyle{amsplain} 
\bibliography{bibliography}

\end{document}